\documentclass{article}
\usepackage[utf8]{inputenc}
\usepackage[frenchb,english]{babel}
\usepackage{amsmath,amsthm,amsfonts,amssymb,oldgerm}
\usepackage[T1]{fontenc}
\usepackage{amsmath}
\usepackage{titlesec}
\usepackage{geometry}
\usepackage{amssymb}
\usepackage{pifont}
\usepackage{dsfont}
\usepackage{verbatim}
\usepackage{graphicx}
\usepackage{color}
\usepackage{mathrsfs}
\usepackage[citecolor=blue,colorlinks=true]{hyperref}
\numberwithin{equation}{section}

\newcommand{\R}{\mathbb{R}}
\newcommand{\T}{\mathbb{T}}

\newcommand{\N}{\mathbb{N}}
\newcommand{\PP}{\mathbb{P}}
\newcommand{\E}{\mathbb{E}}

\newcommand{\MI}{M^{\!-\! 1}\!I(n)}
\newcommand{\eps}{\varepsilon}
\newcommand{\LL}{L^2_{F^{-1}}}
\newcommand{\dd}{\mathrm{d}}
\newcommand{\fe}{f^{\eps}}
\newcommand{\gep}{g^{\eps}}
\newcommand{\re}{\rho^{\eps}}
\newcommand{\li}[1]{\overline{#1}}

\newtheorem{theorem}{Theorem}[section]
\newtheorem{lemma}[theorem]{Lemma}
\newtheorem{prop}[theorem]{Proposition}

\newtheorem{definition}{Definition}[section]

\newenvironment{remark}[1][Remark]{\begin{trivlist}
\item[\hskip \labelsep {\bfseries #1}]}{\end{trivlist}}

\begin{document}
\begin{center}
\textbf{DIFFUSION LIMIT FOR THE RADIATIVE TRANSFER EQUATION PERTURBED BY A MARKOVIAN PROCESS}
\end{center}

\vspace{0.2cm}
\begin{center}
\small \sc{A. Debussche\footnote[1]{\label{ref}IRMAR, ENS Rennes, CNRS, UEB. av Robert Schuman, F-35170 Bruz, France. Email: arnaud.debussche@ens-rennes.fr; sylvain.demoor@ens-rennes.fr}, S. De Moor\textsuperscript{\ref{ref}} and J. Vovelle\footnote[2]{Universit\'e de Lyon ; CNRS ; Universit\'e Lyon 1, Institut Camille Jordan, 43 boulevard du 11 novembre 1918, F-69622 Villeurbane Cedex, France. Email: vovelle@math.univ-lyon1.fr}}
\end{center}

\vspace{0.2cm}

\begin{abstract}
\footnotesize We study the stochastic diffusive limit of a kinetic radiative transfer equation, which is non linear, involving a small parameter and perturbed by a smooth random term. Under an appropriate scaling for the small parameter, using a generalization of the perturbed test-functions method, we show the convergence in law to a stochastic non linear fluid limit.\\

\end{abstract}

\normalsize

\section{Introduction}
In this paper, we are interested in the following non-linear equation
\begin{equation}\label{rt}
\left\{
\begin{aligned}
& \;\; \partial_t \fe\ +\ \frac{1}{\eps}a(v)\cdot\nabla_x \fe\ =\ \frac{1}{\eps^2}\sigma(\li{\fe})L(\fe)\ +\ \frac{1}{\eps}\fe m^{\eps}, \\ 
& \;\; \fe(0)= \fe_0, \qquad t\in[0,T],\, x\in \T^N,\,v\in V.
\end{aligned} 
\right.
\end{equation}
where $(V,\mu)$ is a measured space, $a:V\to \R^N$, $\sigma:\R\to\R$ . The notation $\li{f}$ stands for the average over the velocity space $V$ of the function $f$, that is 
$$\li{f}=\int_V f\,\dd \mu (v).$$
The operator $L$ is a linear operator of relaxation which acts on the velocity variable $v\in V$ only. It is given by
\begin{equation}\label{defL}
L(f):=\li{f}F-f,
\end{equation}
where $v \mapsto F(v)$ is a velocity equilibrium function such that
\begin{equation}\label{propF}
F>0 \text{ a.s.},\quad \li{F}=1,\quad \sup\limits_{v\in V} F(v) < \infty.
\end{equation}
The term $m^{\eps}$ is a random process depending on $(t,x)\in\R^+\times \R^N$ (see section \ref{sectionmeps}). The precise description of the problem setting will be given in the next section. In this paper, we study the behaviour in the limit $\eps\to 0$ of the solution $\fe$ of $(\ref{rt})$. \medskip

\noindent Concerning the physical background in the deterministic case ($m^{\eps}\equiv 0$), equation \eqref{rt} describes the interaction between a surrounding continuous medium and a flux of photons radiating through it in the absence of hydrodynamical motion. The unknown $f^{\eps}(t,x,v)$ then stands for a distribution function of photons having position $x$ and velocity $v$ at time $t$. The function $\sigma$ is the opacity of the matter. When the surrounding medium becomes very large compared to the mean free paths $\eps$ of photons, the solution $\fe$ to \eqref{rt} is known to behave like $\rho F$ where $\rho$ is the solution of the Rosseland equation 
$$
\partial_t \rho - \mathrm{div}_x(\sigma(\rho)^{-1}K\nabla_x\rho)\ =\ 0, \qquad (t,x)\in  [0,T] \times \T^N,
$$
and $F$ is the velocity equilibrium defined above. This is what we call the Rosseland approximation. In this paper, we investigate such an approximation where we have perturbed the deterministic equation by a smooth multiplicative random noise. To do so, we use the method of perturbed test-functions. This method provides an elegant way of deriving stochastic diffusive limit from random kinetic systems; it was first introduced by Papanicolaou, Stroock and Varadhan \cite{psv}. The book of Fouque, Garnier, Papanicolaou and Solna \cite{fgps} presents many applications to this method. A generalization in infinite dimension of the perturbed test-functions method arose in recent papers of Debussche and Vovelle \cite{arnaudjulien} and de Bouard and Gazeau \cite{debouard}.  \medskip

\noindent In the deterministic case (that is when $m^{\eps}\equiv 0$), the Rosseland approximation has been widely studied. In the paper of Bardos, Golse and Perthame \cite{bardos87}, they derive the Rosseland approximation on a slightly more general equation of radiative transfer type than \eqref{rt} where the solution also depends on the frequency variable $\nu$. Using the so-called Hilbert's expansion method, they prove a strong convergence of the solution of the radiative transfer equation to the solution of the Rosseland equation. In \cite{bardos88}, the Rosseland approximation is proved in a weaker sense with weakened hypothesis on the various parameters of the radiative transfer equation, in particular on the opacity function $\sigma$. \medskip

\noindent In the stochastic setting, the case where $\sigma\equiv\sigma_0$ is constant has been studied in the paper of Debussche and Vovelle \cite{arnaudjulien} where they prove the convergence in law of the solution of \eqref{rt} to a limit stochastic fluid equation by mean of a generalization of the perturbed test-functions method. Thus the radiative transfer equation \eqref{rt} is a first step in studying approximation diffusion on non-linear stochastic kinetic equations since the operator $\sigma(\li{f})Lf$ stands for a simple non-linear perturbation of the classical linear relaxation operator $L$. \medskip

\noindent As expected, we have to handle some difficulties caused by this non-linearity. In the paper of Debussche and Vovelle \cite{arnaudjulien} is proved the tightness of the family of processes $(\re)_{\eps>0}$ in the space of time-continuous function with values in some negative Sobolev space $H^{-\eta}(\T^N)$. In our non-linear setting, this is not any more sufficient to succeed in passing to the limit as $\eps$ goes to $0$. As a consequence, the main step to overcome this difficulty is to prove the tightness of the family of processes $(\re)_{\eps>0}$ in the space $L^2(0,T;L^2(\T^N))$. This is made using averaging lemmas in the $L^2$ setting with a slight adaptation to our stochastic context. The main results about deterministic averaging lemmas that we will use in the sequel can be found in the paper of Jabin \cite{jabin}. We point out that, thanks to this additional tightness result, we could handle the case of a more general and non-linear noise term in \eqref{rt} of the form $\frac{1}{\eps}m^{\eps}\lambda(\li{\fe})\fe$ where $\lambda:\R\to\R$ is a bounded and continuous function. In particular, this remains valid in the linear case $\sigma\equiv 1$ studied in the paper \cite{arnaudjulien} of Debussche and Vovelle so that this paper can provide some improvements to their result. \medskip

\noindent {\em Aknowledgements:} This work is partially supported by the french government thanks to the ANR program Stosymap. It also benefit from the support of the french government ``Investissements d'Avenir'' program ANR-11-LABX-0020-01.

\section{Preliminaries and main result}
\subsection{Notations and hypothesis}\label{sec:nothyp}
Let us now introduce the precise setting of equation \eqref{rt}. We work on a finite-time interval $[0,T]$ where $T>0$ and consider periodic boundary conditions for the space variable: $x\in\T^N$ where $\T^N$ is the $N$-dimensional torus. Regarding the velocity space $V$, we assume that $(V,\mu)$ is a measured space. \medskip

\noindent In the sequel, $\LL$ denotes the $F^{-1}$ weighted $L^2(\T^N\times V)$ space equipped with the norm  
$$\|f\|^2:=\int_{\T^N}\!\int_{V}\frac{|f(x,v)|^2}{F(v)}\,\dd \mu(v) \dd x.$$
We denote its scalar product by $(.,.)$. We also need to work in the space $L^2(\T^N)$, which will be often written $L^2$ for short when the context is clear. In what follows, we will often use the inequality  
$$\|\li{f}\|_{L^2_x}\leq\|f\|,$$
which is just Cauchy-Schwarz inequality and the fact that $\li{F}=1$.  We also introduce the Sobolev spaces on the torus $H^{\gamma}(\T^N)$, or $H^{\gamma}$ for short. For $\gamma\in\N$, they consist of periodic functions which are in $L^2(\T^N)$ as well as their derivatives up to order $\gamma$. For general $\gamma\geq 0$, they are easily defined by Fourier series. For $\gamma<0$, $H^{\gamma}(\T^N)$ is the dual of $H^{-\gamma}(\T^N)$.\medskip

\noindent Concerning the velocity mapping $a:V\to \R^N$, we shall assume that it is bounded, that is
\begin{equation}\label{abound}
\sup\limits_{v\in V}|a(v)|<\infty.
\end{equation}
Furthermore, we suppose that the following null flux hypothesis holds
\begin{equation}\label{nullflux}
\int_Va(v)F(v)\,\dd \mu(v)=0,
\end{equation}
and that the following matrix
$$K:=\int_V a(v)\otimes a(v)F(v)\,\dd \mu(v)$$
is definite positive. Finally, to obtain some compactness in the space variable by means of averaging lemmas, we also assume the following standard condition:
\begin{equation}\label{nondegenlemmemoy}
\forall \eps >0,\, \forall (\xi,\alpha)\in  S^{N-1} \!\!\times \R,\;\; \mu\left(\{v\in V, |a(v)\cdot\xi + \alpha|<\eps\}\right)\leq \eps^{\theta},
\end{equation}
for some $\theta\in (0,1]$. \medskip

\noindent Let us now give several hypothesis on the opacity function $\sigma:\R\to\R$. We assume that
\begin{enumerate}
\item[ (H1)] There exist two positive constants $\sigma_*$, $\sigma^* >0$ such that for almost all $x\in\R$, we have 
$$\sigma_*\leq \sigma(x)\leq \sigma^*;$$ 
\item[ (H2)] the function $\sigma$ is Lipschitz continuous.
\end{enumerate}

\noindent Similarly as in the deterministic case, we expect with $(\ref{rt})$ that $\sigma(\li{\fe}) L(\fe)$ tends to zero with $\eps$, so that we should determine the equilibrium of the operator $\sigma(\li{\cdot}) L(\cdot)$. In this case, since $\sigma>0$, they are clearly constituted by the functions of the form $\rho F$ with $\rho$ being independent of $v\in V$. Note that it can easily be seen that $\sigma(\li{\cdot}) L(\cdot)$ is a bounded operator from $\LL$ to $\LL$ and that it is dissipative; precisely, for $f\in \LL $,
\begin{equation}\label{dissip}
(\sigma(\li{f})Lf,f)=-\|\sigma^{\frac{1}{2}}(\li{f})Lf\|^2\leq 0.
\end{equation}
In the sequel, we denote by $g(t,\cdot)$ the semi-group generated by the operator $\sigma(\li{\cdot}) L(\cdot)$ on $\LL$. It verifies, for $f\in\LL$,
$$
\left\{
\begin{aligned}
\frac{d}{dt}g(t,f)&=\sigma(\li{g(t,f)})Lg(t,f), \\
g(0,f)&=f,
\end{aligned}
\right.
$$
and we can show that it is given by
$$g(t,f)=\li{f}F+(f-\li{f}F)e^{-t\sigma(\li{f})},\quad t\geq 0,\; f\in\LL.$$
With the hypothesis (H1) made on $\sigma$, we deduce the following relaxation property of the operator $\sigma(\li{\cdot}) L(\cdot)$
\begin{equation}\label{relaxation}
g(t,f)\longrightarrow \li{f}F,\quad t\to\infty,\qquad \text{ in }\LL.
\end{equation}

\subsection{The random perturbation}\label{sectionmeps}

The random term $m^{\eps}$ is defined by 
$$m^{\eps}(t,x):=m\left(\frac{t}{\eps^2},x\right),$$ 
where $m$ is a stationary process on a probability space $(\Omega,\mathcal{F},\PP)$ and is adapted to a filtration $(\mathcal{F}_t)_{t\geq 0}$. Note that $m^{\eps}$ is adapted to the filtration $(\mathcal{F}^{\eps}_t)_{t\geq 0}=(\mathcal{F}_{\eps^{-2}t})_{t\geq 0}.$\medskip

\noindent We assume that, considered as a random process with values in a space of spatially dependent functions, $m$ is a stationary homogeneous Markov process taking values in a subset $E$ of $W^{1,\infty}(\T^N)$. In the sequel, $E$ will be endowed with the norm $\|\cdot\|_{\infty}$ of $L^{\infty}(\T^N)$. Besides, we denote by $\mathcal{B}(E)$ the set of bounded functions from $E$ to $\R$ endowed with the norm $\|g\|_{\infty}:=\sup_{n\in E}|g(n)|$ for $g\in\mathcal{B}(E)$.  \medskip

\noindent We assume that $m$ is stochastically continuous. Note that $m$ is supposed not to depend on the variable $v$. For all $t\geq 0$, the law $\nu$ of $m_t$ is supposed to be centered
$$\E m_t=\int_E n\,\dd \nu(n)=0.$$
We denote by $e^{tM}$ a transition semi-group on $E$ associated to $m$ and by $M$ its infinitesimal generator. $\mathrm{D}(M)$ stands for the domain of $M$; it is defined as follows:
$$\mathrm{D}(M):=\left\{u\in \mathcal{B}(E), \,\lim\limits_{h\to 0}\frac{e^{hM}-I}{h}u \textrm{ exists in } \mathcal{B}(E)\right\},$$
and if $u\in\mathrm{D}(M)$, we have
$$Mu:=\lim\limits_{h\to 0}\frac{e^{hM}-I}{h}u \textrm{ in } \mathcal{B}(E).$$
Moreover, we suppose that $m$ is ergodic and satisfies some mixing properties in the sense that there exists a subspace $\mathscr{P}_M$ of $\mathcal{B}(E)$ such that for any $g\in\mathscr{P}_M$, the Poisson equation
$$M\psi=g-\int_Eg(n)\,\dd \nu(n)=:\widehat{g},$$
has a unique solution $\psi\in\mathrm{D}(M)$ satisfying $\int_E\psi(n)\,d\nu(n)=0$. We denote by $M^{-1}\widehat{g}$ 
this unique solution, and assume that it is given by
\begin{equation}\label{inversedeM}
M^{-1}\widehat{g}(n)= -\int_0^{\infty}e^{tM}\widehat{g}(n)\dd t,\quad n\in E.
\end{equation}
In particular, we suppose that the above integral is well defined. We need that $\mathscr{P}_M$ contains sufficiently many functions. Thus we assume that for all $f,g \in\LL$, we have
\begin{equation}\label{psi1}
\psi^{(1)}_{f,g}:n\mapsto(fn,g) \in \mathscr{P}_M,
\end{equation}
and we then define $M^{\!-\! 1}\!I$ from $E$ into $W^{1,\infty}(\T^N)$ by
\begin{equation}\label{MI}
(f\MI,g):=M^{\!-\! 1}\psi^{(1)}_{f,g}(n),\quad \forall f,g\in\LL.
\end{equation}
Then, we also suppose that for all $f,g,h \in\LL$ and all continuous operator $B$ from $\LL$ to the space of the continuous bilinear operators on $\LL\times\LL$,
\begin{equation}\label{psi23}
\psi^{(2)}_{f,g}:n\mapsto(fn\MI,g),\quad \psi^{(3)}_{B,f,g,h}:n\mapsto B(f)(gn,h\MI) \in \mathscr{P}_M.
\end{equation}
We need a uniform bound in $W^{1,\infty}(\T^N)$ of all the functions of the variable $n\in E$ introduced above. Namely, we assume, for all $f,g \in\LL$ and all continuous operator $B$ on $\LL$,
\begin{equation}\label{mbound}
\begin{array}{ll}
\|n\|_{W^{1,\infty}(\T^N)} \leq C_*, & \|\MI\|_{W^{1,\infty}(\T^N)} \leq C_*,\\
|M^{\!-\! 1}\!\psi_{f,g}^{(2)}|\leq C_*\|f\|\|g\|, & |M^{\!-\! 1}\!\psi_{B,f,g}^{(3)}|\leq C_*\|B(f)\|\|f\|\|g\|.
\end{array}
\end{equation}
Finally, we suppose that for all $f,g\in\LL$, 
\begin{equation}\label{MIcarre}
n\mapsto (f\MI,g)^2 \in \mathrm{D}(M) \text{ with }|M[(f\MI,g)^2]|\leq C_*\|f\|^2\|g\|^2.
\end{equation}
\noindent To describe the limiting stochastic partial differential equation, we then set 
$$k(x,y)=\E\int_{\R}m_0(y)m_t(x)\,dt,\quad x,y\in\T^N.$$
We can easily show that the kernel $k$ belong to $L^{\infty}(\T^N\times\T^N)$ and, $m$ being stationary, that it is symmetric (see \cite{arnaudjulien}). As a result, we introduce the operator $Q$ on $L^2(\T^N)$ associated to the kernel $k$ 
$$Qf(x)=\int_{\T^N}k(x,y)f(y)\,\dd y,$$
which is self-adjoint, compact and non-negative (see \cite{arnaudjulien}). As a consequence, we can define the square root $Q^{\frac{1}{2}}$ which is Hilbert-Schmidt on $L^2(\T^N)$.
\begin{remark} The above assumptions on the process $m$ are verified, for instance, when $m$ is a Poisson process taking values in a bounded subset $E$ of $W^{1,\infty}(\T^N)$.
\end{remark} 

\subsection{Resolution of the kinetic equation}
In this section, we solve the linear evolution problem \eqref{rt} thanks to a semi-group approach. We thus introduce the linear operator $A:=a(v)\cdot\nabla_x$ on $\LL$ with domain $$\mathrm{D}(A):=\{f\in \LL, \nabla_xf\in \LL \}.$$
The operator $A$ has dense domain and, since it is skew-adjoint, it is $m$-dissipative. Consequently $A$ generates a contraction semigroup $(\mathcal{T}(t))_{t\geq 0}$ (see \cite{cazhar}). We recall that $\mathrm{D}(A)$ is endowed with the norm $\|\cdot\|_{\mathrm{D}(A)}:=\|\cdot\|+\|A\cdot\|$, and that it is a Banach space.

\begin{prop}\label{solkinetic}Let $T>0$ and $f_0^{\eps}\in L^2_{F^{-1}}$. Then there exists a unique mild solution of \eqref{rt}  on $[0,T]$ in $L^{\infty}(\Omega)$, that is there exists a unique $\fe\in L^{\infty}(\Omega,C([0,T],\LL))$ such that $\PP-$a.s. 
$$\fe_t=\mathcal{T}\left(\frac{t}{\eps}\right)\fe_0+\int_0^t\mathcal{T}\left(\frac{t-s}{\eps}\right)\left(\frac{1}{\eps^2}\sigma(\li{\fe_s})L\fe_s+\frac{1}{\eps}m^{\eps}_s \fe_s \right)\,ds,\quad t\in[0,T].$$
Assume further that $\fe_0\in \mathrm{D}(A)$, then there exists a unique strong solution $\fe$ which belongs to the spaces $ L^{\infty}(\Omega,C^1([0,T],\LL))$ and $L^{\infty}(\Omega,C([0,T],\mathrm{D}(A)))$ of \eqref{rt}.
\end{prop}
\begin{proof}
Subsections 4.3.1 and 4.3.3 in \cite{cazhar} gives that $\PP-$a.s. there exists a unique mild solution $\fe\in C([0,T],\LL)$ and it is not difficult to slightly modify the proof to obtain that in fact $\fe\in L^{\infty}(\Omega,C([0,T],\LL))$ (we intensively use that for all $t\geq 0$ and $\eps>0$, $\|m^{\eps}_t\|_{W^{1,\infty}(\T^N)}\leq C_*$). \\
Similarly, subsections 4.3.1 and 4.3.3 in \cite{cazhar} gives us $\PP-$a.s. a strong solution $\fe$ in the spaces $C^1([0,T],\LL)$ and $C([0,T],\mathrm{D}(A))$ of \eqref{rt} and once again one can easily get that in fact $\fe$ belongs to the spaces $L^{\infty}(\Omega,C^1([0,T],\LL))$ and $L^{\infty}(\Omega,C([0,T],\mathrm{D}(A)))$. 
\end{proof}

\begin{remark} If $f^{\eps}_0\in \mathrm{D}(A)$, we thus have, for $\eps>0$ fixed,
\begin{equation}\label{bornew}
\sup\limits_{t\in[0,T]}\|\fe_t\|+\sup\limits_{t\in[0,T]}\|Af^{\eps}_t\| \in L^{\infty}(\Omega).
\end{equation}
\end{remark}

\subsection{Main result}
We are now ready to state our main result.
\begin{theorem}\label{mainresult}
Assume that $(f^{\eps}_0)_{\eps>0}$ is bounded in $\LL$ and that 
$$\rho^{\eps}_0:=\int_Vf^{\eps}_0\,\dd \mu(v)\underset{\eps\to 0}{\longrightarrow} \rho_0 \text{ in }L^2(\T^N).$$
Then, for all $\eta>0$ and $T>0$, $\re:=\li{\fe}$ converges in law in $C([0,T],H^{-\eta}(\T^N))$ and $L^2(0,T;L^2(\T^N))$ to the solution $\rho$ to the non-linear stochastic diffusion equation
\begin{equation}\label{stochasticeq}
\dd \rho - \mathrm{div}_x(\sigma(\rho)^{-1}K\nabla_x\rho)\ \dd t\ =\ H\rho\ \dd t + \rho\, Q^{\frac{1}{2}}\dd W_t, \text{ in } [0,T] \times \T^N,
\end{equation}
with initial condition $\rho(0)=\rho_0$ in $L^2(\T^N)$, and where $W$ is a cylindrical Wiener process on $L^2(\T^N)$, \begin{equation}
\label{defK}
K:=\int_V a(v)\otimes a(v)F(v)\,\dd \mu(v)
\end{equation}
and 
\begin{equation}
\label{defH}
H:=\int_E nM^{-1}I(n)\,\dd \nu(n)\in W^{1,\infty}.
\end{equation}
\end{theorem}
\begin{remark} The limit equation $(\ref{stochasticeq})$ can also be written in Stratonovich form
$$
\dd \rho - \mathrm{div}_x(\sigma(\rho)^{-1}K\nabla_x\rho)\ \dd t\ =\ \rho\circ Q^{\frac{1}{2}}\dd W_t.
$$
\end{remark}

\noindent \textbf{Notation} In the sequel, we denote by $\lesssim$ the inequalities which are valid up to constants of the problem, namely $C_*$, $N$, $\sup_{\eps>0}\|\fe_0\|$, $\sup_{v\in V}|a(v)|$, $\sup_{v\in V}F(v)$, $\sigma_*$, $\sigma^*$, $\|\sigma\|_{\text{Lip}}$  and real constants.

\section{The generator}
The process $f^{\eps}$ is not Markov (indeed, by \eqref{rt}, we need $m^{\eps}$ to know the increments of $f^{\eps}$) but the couple $(f^{\eps},m^{\eps})$ is. From now on, we denote by $\mathscr{L}^{\eps}$ its infinitesimal generator, that is
$$\mathscr{L}^{\eps}\varphi(f,n):=\lim_{h\to 0}\frac{1}{h}\E\left[\varphi(f^{\eps}_h,m^{\eps}_h)-\varphi(f,n)\big{|}(f^{\eps}_0,m^{\eps}_0)=(f,n)\right],$$
where $\varphi:\LL\times E\to\R$ belongs to the domain of $\mathscr{L}^{\eps}$. Thus we begin this section by introducing a special set of functions which lie in the domain of  $\mathscr{L}^{\eps}$ and satisfy the associated martingale problem. \\ \\
In the following, if $\varphi:\LL\rightarrow\R$ is differentiable with respect to $f\in\LL$, we denote by $D\varphi(f)$ its differential at a point $f$ and we identify the differential with the gradient. 
\begin{definition}\label{goodtest}We say that $\varphi:\LL\times E\rightarrow\R$ is a good test function if 
\begin{enumerate}
\item[$(i)$]$(f,n)\mapsto\varphi(f,n)$ is differentiable with respect to $f$;
\item[$(ii)$]$(f,n)\mapsto D\varphi(f,n)$ is continuous from $\LL\times E$ to $\LL$ and maps bounded sets onto bounded sets;
\item[$(iii)$]for any $f\in\LL$, $\varphi(f,\cdot)\in D_M$;
\item[$(iv)$]$(f,n)\mapsto M\varphi(f,n)$ is continuous from $\LL\times E$ to $\R$ and maps bounded sets onto bounded sets.
\end{enumerate}
\end{definition}

\begin{prop}\label{gene}Let $\varphi$ be a good test function. Then, for all $(f,n)\in\mathrm{D}(A)\times E$, 
$$\mathscr{L}^{\eps}\varphi(f,n)=-\frac{1}{\eps}(Af,D\varphi(f))+\frac{1}{\eps^2}(\sigma(\li{f})Lf,D\varphi(f))+\frac{1}{\eps}(fn,D\varphi(f))+\frac{1}{\eps^2}M\varphi(f,n).$$
Furthermore, if $f^{\eps}_0\in\mathrm{D}(A)$,
$$M^{\eps}_{\varphi}(t):=\varphi(f^{\eps}_t,m^{\eps}_t)-\varphi(f^{\eps}_0,m^{\eps}_0)-\int_0^t\mathscr{L}^{\eps}\varphi(f^{\eps}_s,m^{\eps}_s)\,ds$$
is a continuous and integrable $(\mathcal{F}^{\eps}_t)_{t\geq 0}$ martingale, and if $|\varphi|^2$ is a good test function, its quadratic variation is given by
$$\langle M^{\eps}_{\varphi}\rangle_t=\int_0^t(\mathscr{L}^{\eps}|\varphi|^2-2\varphi\mathscr{L}^{\eps}\varphi)(f^{\eps}_s,m^{\eps}_s)\,ds.$$
\end{prop}
\begin{proof}We compute the expression of the infinitesimal generator as follows :
\begin{alignat*}{2}
\mathscr{L}^{\eps}\varphi(f,n)&=\lim_{h\to 0}\frac{1}{h}\E\left[\varphi(f^{\eps}_h,m^{\eps}_h)-\varphi(f,n)\big{|}(f^{\eps}_0,m^{\eps}_0)=(f,n)\right]\\
&=\lim_{h\to 0}\frac{1}{h}\E\left[\varphi(f^{\eps}_h,m^{\eps}_h)-\varphi(f,m^{\eps}_h)\big{|}(f^{\eps}_0,m^{\eps}_0)=(f,n)\right]\\
&+\lim_{h\to 0}\frac{1}{h}\E\left[\varphi(f,m^{\eps}_h)-\varphi(f,n)\big{|}m^{\eps}_0=n\right]
\end{alignat*}
Since $\varphi$ verifies point $(iii)$ of Definition \ref{goodtest}, the second term of the last equality goes to $\eps^{-2}M\varphi(f,n)$ when $h\to 0$. We now focus on the first term. With points $(i)-(ii)$ of Definition \ref{goodtest}, we have that $\varphi$ is continuously differentiable with respect to $f$. Thus 
$$\varphi(f^{\eps}_h,m^{\eps}_h)-\varphi(f,m^{\eps}_h)=\int_0^1D\varphi(f+s(f^{\eps}_h-f),m^{\eps}_h)(f^{\eps}_h-f)\,\dd s.$$ Besides, since $f^{\eps}_0=f\in\mathrm{D}(A)$, $f^{\eps}\in C^1([0,T],\LL)$ and we have $$f^{\eps}_h-f=h\int_0^1\partial_tf^{\eps}_{uh}\,\dd u.$$
Thus, we can rewrite the first term as
\begin{alignat*}{2}
&=\lim_{h\to 0}\frac{1}{h}\E\left[\varphi(f^{\eps}_h,m^{\eps}_h)-\varphi(f,m^{\eps}_h)\big{|}(f^{\eps}_0,m^{\eps}_0)=(f,n)\right]\\
&=\lim_{h\to 0}\E_{(f,n)}\left[\int_0^1\!\!\!\int_0^1a_h(w,s,u)\,\dd u\,\dd s\right],
\end{alignat*}
with $a_h(w,s,u):=D\varphi(f+s(f^{\eps}_h-f),m^{\eps}_h)(\partial_tf^{\eps}_{uh})$ and where $\E_{(f,n)}$ denotes the expectation under the probability measure $\PP_{(f,n)}:=\PP(\;\cdot\;|(f_0^{\eps},m_0^{\eps})=(f,n))$.\\

\noindent Recall that $D\varphi$ is continuous with respect to $(f,n)$ thanks to point $(ii)$ of Definition \ref{goodtest}, that $f^{\eps}$ is $\PP-$a.s. in $C^1([0,T],\LL)$ and that $m^{\eps}$ is stochastically continuous to conclude that $a_h$ converges in probability as $h\to 0$ to $D\varphi(f,n)(\partial_tf^{\eps}(0))$ in the probability space $\tilde{\Omega}:=(\Omega\times [0,1]\times [0,1],\PP_{(f,n)}\otimes \dd x\otimes \dd s)$. Furthermore, we prove that $(a_h)_{0\leq h\leq 1}$ is uniformly integrable in $\tilde{\Omega}$ since it is uniformly bounded with respect to $0\leq h\leq 1$ in $L^{\infty}(\tilde{\Omega})$. Indeed, with the fact that $L$ is a bounded operator, with (H1) and the fact that $\|n\|_{L^{\infty}(\T^N)}\lesssim 1$ for all $n\in E$, we get
$$
|a_h|\lesssim \|D\varphi(f+s(f^{\eps}_h-f),m^{\eps}_h)\|(\|f^{\eps}_{uh}\|+\|Af^{\eps}_{uh}\|).$$
With $(\ref{bornew})$, we set 
$$R:=\sup\limits_{t\in[0,T]}\|f^{\eps}_t\|+\sup\limits_{t\in[0,T]}\|Af^{\eps}_t\|\in L^{\infty}(\Omega),$$
and define $r := \|R\|_{L^{\infty}(\Omega)}$.
Then, since $D\varphi$ maps bounded sets on bounded sets, we can bound the term $||D\varphi(f+s(f^{\eps}_h-f),m^{\eps}_h)||$ by $$C:=\sup\left\{\|D\varphi(f,n)\|, f\in B_{\LL}\!\!(0,\|f\|+r),n\in B_{E}(0,C_*)\right\}.$$
So we are led to
$$
\|a_h\|_{L^{\infty}(\tilde{\Omega})}\lesssim C\cdot r,$$
which is what we announced.
To prove the sequel of the proposition, we use the same kind of ideas and follow the proofs of \cite[Proposition 6]{arnaudjulien} and \cite[Appendix 6.9]{fgps}. 
\end{proof}

\section{The limit generator}

In this section, we study the limit of the generator $\mathscr{L}^{\eps}$ when $\eps\to 0$. The limit generator $\mathscr{L}$ will characterize the limit stochastic fluid equation.

\subsection{Formal derivation of the corrections}\label{subsec:formaltest}
\noindent To derive the diffusive limiting equation, one has to study the limit as $\eps$ goes to $0$ of quantities of the form $\mathscr{L}^{\eps}\varphi$ where $\varphi$ is a good test function. To do so, following the perturbed test-functions method, we have to correct $\varphi$ so as to obtain a non-singular limit. We search the correction $\varphi^{\eps}$ of $\varphi$ under the classical form:
$$\varphi^{\eps}:=\varphi+\eps\varphi_1+\eps^2\varphi_2.$$
In this decomposition, $\varphi_1$ and $\varphi_2$ are respectively the first and second order corrections and are to be defined in the sequel so that
$$\mathscr{L}^{\eps}\varphi^{\eps}=\mathscr{L}\varphi+O(\eps),$$
where $\mathscr{L}$ will be the limit generator. We restrict our study to smooth test-functions. Precisely, we introduce the set of spatial derivative operators up to order $3$:
$$\mathcal{R}:=\{\partial^{e_1}_{i_1}\partial^{e_2}_{i_2}\partial^{e_3}_{i_3},\,e\in\{0,1\}^3,\,i\in \{1,...,N\}^3, \,|i|\leq 3\}$$
and we suppose that the test-function $\varphi$ is a good test, that $\varphi\in C^3(\LL)$ and that there exists a constant $C_{\varphi}>0$ such that
\begin{equation}\label{smoothtest}
\left\{
\begin{aligned}
& |\varphi(f)|\leq C_{\varphi}(1+\|f\|^2), \\
& \|\Lambda D\varphi(f)\|\leq C_{\varphi}(1+\|f\|), \\
& |D^2\varphi(f)(\Lambda_1h,\Lambda_2k)|\leq C_{\varphi}\|h\|\|k\|, \\
& |D^3\varphi(f)(\Lambda_1h,\Lambda_2k,\Lambda_3l)|\leq C_{\varphi}\|h\|\|k\|\|l\|,
\end{aligned}
\right.
\end{equation}
for any $f,h,k,l\in \LL$ and $\Lambda,\Lambda_1,\Lambda_2,\Lambda_3 \in \mathcal{R}$. Thanks to Proposition \ref{gene}, and since $\varphi$ does not depend on $n\in E$, we can write
\begin{align}
\mathscr{L}^{\eps}\varphi^{\eps}(f,n)&\label{formal1}=-\frac{1}{\eps}(Af,D\varphi(f))+\frac{1}{\eps^2}(\sigma(\li{f})Lf,D\varphi(f))+\frac{1}{\eps}(fn,D\varphi(f))\\
&\label{formal2}-(Af,D\varphi_1(f))+\frac{1}{\eps}(\sigma(\li{f})Lf,D\varphi_1(f))+(fn,D\varphi_1(f))+\frac{1}{\eps}M\varphi_1\\
&\label{formal3}-\eps(Af,D\varphi_2(f))+(\sigma(\li{f})Lf,D\varphi_2(f))+\eps(fn,D\varphi_2(f))+M\varphi_2.
\end{align}
In the sequel, we do not care about the terms relative to the transport part $A$ of the equation since these terms will be handled as in the deterministic case (when $m^{\eps}\equiv 0$). To be more precise, and as it will be shown in the sequel, the first term of $(\ref{formal1})$ will give rise, as $\eps$ goes to $0$, to the deterministic term in the limit generator $\mathscr{L}$ and the first terms of $(\ref{formal2})$ and $(\ref{formal3})$ are respectively of orders $\eps$ and $\eps^2$. For the remaining terms, in a first step, we would like to cancel those who have a singular power of $\eps$. Thus we should impose that the two following equations hold:
\begin{equation}\label{formaleq1}
(\sigma(\li{f})Lf,D\varphi(f))=0,
\end{equation}
\begin{equation}\label{formaleq2}
(\sigma(\li{f})Lf,D\varphi_1(f))+M\varphi_1+(fn,D\varphi(f)) = 0.
\end{equation}
Let us say a word about the fact that we chose to handle the terms relative to the transport part of the equation separately. When trying to correct these terms thanks to the correctors $\varphi_1$ and $\varphi_2$, the non-linearity $\sigma$ implies that the second corrector $\varphi_2$, unless we can write it formally, does not behave properly any more.
\subsubsection{Equation on $\varphi$}
Let us solve $(\ref{formaleq1})$. We recall that $(g(t,f))_{t\geq 0}$ denotes the semigroup of the operator $\sigma(\li{\cdot})L$. Equation $(\ref{formaleq1})$ gives immediately that the map $t\mapsto\varphi(g(t,f))$ is constant. As a result, with \eqref{relaxation},
$$\varphi(f)=\varphi(g(0,f))=\varphi(\varphi(g(\infty,f))=\varphi(\li{f}F),$$
so that $\varphi$ only depends on $\li{f}F$. This implies, for all $h\in \LL$,
\begin{equation}\label{moyenformal}
(h,D\varphi(f))=(\li{h}F,D\varphi(\li{f}F)).
\end{equation}
\subsubsection{Equation on $\varphi_1$}
Next, we solve $(\ref{formaleq2})$. We consider the Markov process $(g(t,f),m(t,n))_{t\geq 0}$. Its generator will be denoted by $\mathscr{M}$. We observe that equation $(\ref{formaleq2})$ rewrites:
$$\mathscr{M}\varphi_1(f,n)=-(fn,D\varphi(f)).$$
This Poisson equation will have a solution if the integral of $(f,n)\mapsto(fn,D\varphi(f))$ over $\LL\times E$ equipped with the invariant measure of the process $(g(t,f),m(t,n))_{t\geq 0}$ is zero. So, we must verify that 
$$\int_E(\li{f}Fn,D\varphi(\li{f}F))\,\dd \nu(n)=0,$$
and this relation does hold since $m$ is centered. As a consequence, if we can prove the existence of the integral, we can write $\varphi_1$ as
$$\varphi_1(f,n)=\int_0^{\infty}\E(g(t,f)m(t,n),D\varphi(g(t,f)))\,\dd t.$$
Then, we use $(\ref{moyenformal})$, $\li{g(t,f)}=\li{f}$ and \eqref{psi1} and \eqref{MI} to obtain
\begin{align*}
\varphi_1(f,n)&=\int_0^{\infty}\E(\li{f}F m(t,n),D\varphi(\li{f}F))\,\dd t = -(\li{f}F \MI,D\varphi(\li{f}F))\\
&=-(f\MI,D\varphi(f)).
\end{align*}
We are now able to state the
\begin{prop}[First corrector]\label{f1} Let $\varphi\in C^3(\LL)$ be a good test-function satisfying $(\ref{smoothtest})$ and depending only on $\li{f}F$. For any $(f,n)\in\LL\times E$, we define the first corrector $\varphi_1$ as
$$\varphi_1(f,n):=-(f\MI,D\varphi(f)).$$
Furthermore, it satisfies the bounds
\begin{equation}\label{f1bounds}
(i)\;\;|\varphi_1(f,n)|\lesssim C_{\varphi}(1+\|f\|)^2, \quad (ii)\;\;\|AD\varphi_1(f,n)\|\lesssim C_{\varphi}(1+\|f\|).
\end{equation}
\end{prop}
\noindent Note that the bounds \eqref{f1bounds} are consequences of \eqref{mbound} and \eqref{smoothtest}.
\subsubsection{Equation on $\varphi_2$}
At this stage, we have 
\begin{equation}\label{afterone}
\begin{aligned}
\mathscr{L}^{\eps}\varphi^{\eps}(f,n)&=-\frac{1}{\eps}(Af,D\varphi(f))+\mathscr{M}\varphi_2+(fn,D\varphi_1(f))\\
&-(Af,D\varphi_1(f))-\eps(Af,D\varphi_2(f))+\eps(fn,D\varphi_2(f)).
\end{aligned}
\end{equation}
Note that the limit of $\mathscr{L}^{\eps}\varphi^{\eps}$ as $\eps$ goes to $0$ does depend on $n\in E$ with the term $(fn,D\varphi_1(f))$. Since the expected limit is $\mathscr{L}\varphi$ where $\varphi$ does not depend on $n$, we have to correct this term to cancel the dependence with respect to $n$ of the limit. This is the aim of the second order correction $\varphi_2$. The right way to do so, given the mixing properties of the operator $\mathscr{M}$, is to subtract the mean value of this term under  the invariant measure of the Markov process $(g(t,f),m(t,n))_{t\geq 0}$ governed by $\mathscr{M}$. We write
\begin{align*}
\mathscr{L}^{\eps}\varphi^{\eps}(f,n)&=-\frac{1}{\eps}(Af,D\varphi(f))+\int_E(\li{f}Fn,D\varphi_1(\li{f}F))\,\dd \nu(n)\\
&+\mathscr{M}\varphi_2+(fn,D\varphi_1(f))-\int_E(\li{f}Fn,D\varphi_1(\li{f}F))\,\dd \nu(n)\\
&-(Af,D\varphi_1(f))-\eps(Af,D\varphi_2(f))+\eps(fn,D\varphi_2(f)),
\end{align*}
and we can now define $\varphi_2$ as the solution of the well-posed Poisson equation
$$\mathscr{M}\varphi_2=-(fn,D\varphi_1(f))+\int_E(\li{f}Fn,D\varphi_1(\li{f}F))\,\dd \nu(n).$$
Note that, thanks to the definition of $\varphi_1$ given above, we can compute
$$(\li{f}Fn,D\varphi_1(\li{f}F))=-(fn\MI,D\varphi(f))-D^2\varphi(f)(f\MI,fn)=:q(f,n)$$
As a result, we easily have the following proposition.
\begin{prop}[Second corrector]\label{f2} Let $\varphi\in C^3(\LL)$ be a good test-function satisfying $(\ref{smoothtest})$ and depending only on $\li{f}F$. For any $(f,n)\in\LL\times E$, we define the second corrector $\varphi_2$ as
$$\varphi_2(f,n):=\E\int_0^{\infty}\left(\int_E(q(\li{f}F,n)\,\dd \nu(n)-q(g(t,f),m(t,n))\right)\,\dd t,$$
which is well defined and satisfies the bounds
\begin{equation}\label{f2bounds}
(i)\;\;|\varphi_2(f,n)|\lesssim C_{\varphi}(1+\|f\|)^2, \quad (ii)\;\;\|AD\varphi_2(f,n)\|\lesssim C_{\varphi}(1+\|f\|).
\end{equation}
\end{prop}
\noindent The existence of $\varphi_2$ is based on \eqref{psi23} and the bounds \eqref{f2bounds} are proved using \eqref{mbound} and \eqref{smoothtest}.
\subsubsection{Summary}
The correctors $\varphi_1$ and $\varphi_2$ being defined as above in Propositions \ref{f1} and \ref{f2}, we are finally led to
\begin{align*}
\mathscr{L}^{\eps}\varphi^{\eps}(f,n)&=-\frac{1}{\eps}(Af,D\varphi(f))+\int_E(\li{f}Fn,D\varphi_1(\li{f}F))\,\dd \nu(n)\\
&-(Af,D\varphi_1(f))-\eps(Af,D\varphi_2(f))+\eps(fn,D\varphi_2(f)).
\end{align*}
We are now able to define the limit generator $\mathscr{L}$ as, for all $\rho\in L^2(\T^N)$,
\begin{multline}\label{limitg}
\mathscr{L}\varphi(\rho):=(\mathrm{div}_x(\sigma(\rho)^{-1}K\nabla_x\rho)F,D\varphi(\rho F))-\int_E(\rho F n\MI,D\varphi(\rho F))\,\dd \nu(n)\\
-\int_ED^2\varphi(\rho F)(\rho F\MI,\rho F n)\,\dd \nu(n),
\end{multline}
and we have shown the following equality
\begin{equation}\label{aftertwo}
\begin{aligned}
\mathscr{L}^{\eps}\varphi^{\eps}(f,n)&=\mathscr{L}\varphi(\li{f})-\frac{1}{\eps}(Af,D\varphi(f))-(\mathrm{div}_x(\sigma(\li{f})^{-1}K\nabla_x\li{f})F,D\varphi(\li{f}F))\\
&-(Af,D\varphi_1(f))-\eps(Af,D\varphi_2(f))+\eps(fn,D\varphi_2(f)).
\end{aligned}
\end{equation}

\section{Uniform bound in $\LL$}
In this section, we prove a uniform estimate of the $\LL$ norm of the solution $f^{\eps}$ with respect to $\eps$. To do so, 
we will again use the perturbed test functions method. The result is the following:
\begin{prop}
Let $p\geq 1$ and $\fe_0\in \mathrm{D}(A)$. We have the two following bounds
\begin{equation}\label{L2bound}
\E\sup\limits_{t\in[0,T]}\|\fe_t\|^p\lesssim 1,
\end{equation}
\begin{equation}\label{L2boundtemps}
\E\left(\int_0^T\|\sigma^{\frac{1}{2}}(\li{\fe_s})L\fe_s\|^2\,\dd s\right)^p \lesssim \eps^{2p}.
\end{equation}
\end{prop}
\begin{proof}
We set, for all $f\in\LL$, $\varphi(f):=\frac{1}{2}\|f\|^2$, which is easily seen to be a good test function. Then, with Proposition  \ref{gene}, the fact that $A$ is skew-adjoint, $(\ref{dissip})$, and the fact that $\varphi$ does not depend on $n\in E$, we get for $f\in\mathrm{D}(A)$ and $n\in E$,
\begin{align*}
\mathscr{L}^{\eps}\varphi(f,n)&=-\frac{1}{\eps}(Af,f) + \frac{1}{\eps^2}(\sigma(\li{f})Lf,f)+\frac{1}{\eps}(fn,f)+\frac{1}{\eps^2}M\varphi(f,n)\\
&=-\frac{1}{\eps^2}\|\sigma^{\frac{1}{2}}(\li{f})Lf\|^2+\frac{1}{\eps}(fn,f).
\end{align*}
The first term has a favourable behaviour for our purpose. The second term is more difficult to control and we correct $\varphi$  thanks to the perturbed test-functions method to get rid of it: we recall the formal computations done in Section \ref{subsec:formaltest} and we set $\varphi_1(f,n)=-(f,\MI f)$ and $\varphi^{\eps}:=\varphi(f,n)+\eps\varphi_1$. We can show that $\varphi_1$ is a good test function with, thanks to Proposition \ref{gene},
\begin{align*}\eps\mathscr{L}^{\eps}\varphi_1(f,n)&=-\frac{2}{\eps}(\sigma(\li{f})Lf,\MI f)-2(Af,\MI f) \\
&\qquad -2(fn,\MI f)-\frac{1}{\eps}(fn,f).
\end{align*}

\noindent As a consequence, we are led to
\begin{align*}\mathscr{L}^{\eps}\varphi^{\eps}(f,n)&=-\frac{1}{\eps^2}\|\sigma^{\frac{1}{2}}(\li{f})Lf\|^2-\frac{2}{\eps}(\sigma(\li{f})Lf,\MI f)-2(Af,\MI f) \\
&\qquad -2(fn,\MI f).
\end{align*}
We use \eqref{mbound} and the hypothesis (H1) made on $\sigma$ to bound the second term:
\begin{align*}
\frac{2}{\eps}(\sigma(\li{f})Lf,\MI f)\ &\leq\ 2C_*(\sigma^*)^{\frac{1}{2}}\eps^{-1}\|\sigma^{\frac{1}{2}}(\li{f})Lf\|\|f\|\\
&\ \leq\ \frac{1}{2\eps^2}\|\sigma^{\frac{1}{2}}(\li{f})Lf\|^2+2C_*^2\sigma^*\|f\|^2.
\end{align*}
Furthermore, for the last two terms, we write
\begin{align*}
-2(Af,\MI f)-2(fn,\MI f) & \ =\ (f^2,A\MI)-2(fn,\MI f) \\
&\  \leq\ \|f\|^2\|a\|_{L^{\infty}(V)}C_*+2C_*^2\|f\|^2.
\end{align*}
To sum up, we have proved that 
\begin{equation}\label{lepsborne}
\mathscr{L}^{\eps}\varphi^{\eps}(f,n)\ \lesssim\  -\frac{1}{2\eps^2}\|\sigma^{\frac{1}{2}}(\li{f})Lf\|^2+\|f\|^2.
\end{equation}

\noindent As in Proposition \ref{gene}, since $\varphi^{\eps}$ is a good test function, we now define
$$M^{\eps}(t):=\varphi^{\eps}(f^{\eps}_t,m^{\eps}_t)-\varphi^{\eps}(f^{\eps}_0,m^{\eps}_0)-\int_0^t\mathscr{L}^{\eps}\varphi^{\eps}(f^{\eps}_s,m^{\eps}_s)\,\dd s,$$
which is a continuous and integrable $(\mathcal{F}^{\eps}_t)_{t\geq 0}$ martingale. By definition of $\varphi$, $\varphi^{\eps}$ and $M^{\eps}$, we obtain
$$\frac{1}{2}\|f^{\eps}_t\|^2=\frac{1}{2}\|f^{\eps}_0\|^2-\eps(\varphi_1(f^{\eps}_t,m^{\eps}_t)-\varphi_1(f^{\eps}_0,m^{\eps}_0))+\int_0^t\mathscr{L}^{\eps}\varphi^{\eps}(f^{\eps}_s,m^{\eps}_s)\,\dd s+M^{\eps}(t).$$
Since we have obviously $|\varphi_1(f,n)|\lesssim\|f\|^2$, we can write, with $(\ref{lepsborne})$, 
$$\|f^{\eps}_t\|^2\ \lesssim\ \|f^{\eps}_0\|^2+\eps\|f^{\eps}_t\|+\int_0^t-\frac{1}{2\eps^2}\|\sigma^{\frac{1}{2}}(\li{\fe_s})L\fe_s\|^2+\|\fe_s\|^2\,\dd s+\sup\limits_{t\in[0,T]}|M^{\eps}(t)|,$$
i.e. for $\eps$ sufficiently small,
$$\int_0^t\frac{1}{2\eps^2}\|\sigma^{\frac{1}{2}}(\li{\fe_s})L\fe_s\|^2\,\dd s+\|f^{\eps}_t\|^2\ \lesssim\ \|f^{\eps}_0\|^2+\int_0^t\|f^{\eps}_s\|^2\,ds+\sup\limits_{t\in[0,T]}|M^{\eps}(t)|,$$
and by Gronwall lemma,
\begin{equation}\label{z1}
\int_0^t\frac{1}{2\eps^2}\|\sigma^{\frac{1}{2}}(\li{\fe_s})L\fe_s\|^2\,\dd s+\|f^{\eps}_t\|^2\ \lesssim\ \|f^{\eps}_0\|^2+\sup\limits_{t\in[0,T]}|M^{\eps}(t)|.
\end{equation}
Note that $|\varphi^{\eps}|^2$ is a good test function with, thanks to \eqref{mbound} and \eqref{MIcarre},
$$|\mathscr{L}^{\eps}|\varphi^{\eps}|^2-2\varphi^{\eps}\mathscr{L}^{\eps}\varphi^{\eps}|=|M|\varphi_1|^2-2\varphi_1M\varphi_1|\lesssim \|f\|^4,$$
and that, with Proposition \ref{gene}, the quadratic variation of $M^{\eps}(t)$ is given by 
$$\langle M^{\eps}\rangle_t=\int_0^t(\mathscr{L}^{\eps}|\varphi^{\eps}|^2-2\varphi^{\eps}\mathscr{L}^{\eps}\varphi^{\eps})(f^{\eps}_s,m^{\eps}_s)\,\dd s.$$
As a result, with Burkholder-Davis-Gundy  and Hölder inequalities, we get
\begin{equation}\label{z2}
\E\sup\limits_{t\in[0,T]}|M^{\eps}(t)|^p\ \lesssim\ \E|\langle M^{\eps}\rangle_T|^{\frac{p}{2}}\lesssim \int_0^T\E\|f^{\eps}_s\|^{2p}\,\dd s.
\end{equation}
Neglecting the first (positive) term of the left-hand side in $(\ref{z1})$, we have 
$$\E\|f^{\eps}_t\|^{2p}\ \lesssim\ \E\|f^{\eps}_0\|^{2p}+\E\sup\limits_{t\in[0,T]}|M^{\eps}(t)|^p,
$$
so that we get 
$$\E\|f^{\eps}_T\|^{2p}\ \lesssim\ \E\|f^{\eps}_0\|^{2p}+\int_0^T\E\|f^{\eps}_s\|^{2p}\,\dd s,
$$
and, by Gronwall lemma,
\begin{equation}\label{z3}
\E\|f^{\eps}_T\|^{2p}\ \lesssim\ \E\|f^{\eps}_0\|^{2p}.
\end{equation}
This actually holds true for any $t\in[0,T]$. Thus, using $(\ref{z2})$ and $(\ref{z3})$ in $(\ref{z1})$ finally gives the expected bounds. 
\end{proof}
\begin{remark}
We define $g^{\eps}:=\fe - \re F = - L\fe$. Since we have $\sigma\geq \sigma_*$, the bound $(\ref{L2boundtemps})$ gives that, for all $p\geq 1$,
\begin{equation}\label{gepsbound}
(\eps^{-1}g^{\eps})_{\eps>0} \text{ is bounded in } L^p(\Omega;L^2(0,T;\LL)).
\end{equation}
\end{remark}
In the sequel, we must deal with the non-linear term. To do so, we need some compactness in the space variable of the process $(\re)_{\eps>0}$. The following proposition is a first step to this purpose.

\begin{prop}
We assume that hypothesis \eqref{nondegenlemmemoy} is satisfied. Let $p\geq 1$ and $s\in(0,\theta/2)$. We have the bound
\begin{equation}\label{Hsbound}
\E\left(\int_0^T\|\re_s\|^2_{H^s(\T^N)}\,\dd s\right)^p \lesssim 1.
\end{equation}
\end{prop}
\begin{proof}
Note that with $\sigma\leq\sigma^*$, the remark $(\ref{gepsbound})$ and equation $(\ref{rt})$, we observe that 
$$
(\eps\partial_t \fe + a(v)\cdot\nabla_x \fe - \fe m^{\eps})_{\eps>0} \text{ is bounded in } L^p(\Omega;L^2(0,T;\LL)).
$$
Furthermore, $(\fe)_{\eps>0}$ is bounded in $L^p(\Omega;L^2(0,T;\LL))$ with $(\ref{L2bound})$ and $|m^{\eps}|\leq C_*$ so that
\begin{equation}\label{epsdtborne}
(\eps\partial_t \fe + a(v)\cdot\nabla_x \fe)_{\eps>0} \text{ is bounded in } L^p(\Omega;L^2(0,T;\LL)).
\end{equation}
Then, thanks to \eqref{nondegenlemmemoy}, we apply an averaging lemma to conclude. Precisely, \cite[Theorem 3.1]{jabin} in the unstationary case applies a.s. with $\beta=\gamma=0$, $p_1=q_1=p_2=q_2=2$, $a=0$, $k=\theta$ and
$$f = \fe,\qquad g=\eps\partial_t \fe + a(v)\cdot\nabla_x \fe,$$
and gives the bound
$$
\|\re\|_{B^{\frac{\theta}{2},2}_{\infty,\infty}}\leq C \|\fe\|^{\frac{1}{2}}\|\eps\partial_t \fe + a(v)\cdot\nabla_x \fe\|^{\frac{1}{2}}\quad \text{a.s.}
$$ Since, for any $s<\theta/2$, $H^{s}\subset B^{\frac{\theta}{2}}_{\infty,\infty}$, it yields, for $p \geq 1$,
$$
\E\left(\int_0^T\|\re_s\|^2_{H^s}\,\dd s\right)^p \leq C \E\left(\int_0^T\|\fe_s\|\|\eps\partial_t \fe_s + a(v)\cdot\nabla_x \fe_s\|\,\dd s\right)^p,
$$
so that the result follows with Cauchy Schwarz inequality and \eqref{L2bound} and \eqref{epsdtborne}. This concludes the proof.
\end{proof}

\section{Tightness}\label{tension}

We want to prove the convergence in law of the family $(\re)_{\eps>0}$: in this section, we study the tightness of the processes  $(\re)_{\eps>0}$ in the space $C([0,T],H^{-\eta}(\T^N))$ where $\eta>0$. In fact, this will not be sufficient to pass to the limit in the non-linear term. As a consequence, we also prove that $(\re)_{\eps>0}$ is tight in the space $L^2(0,T;L^2(\T^N))$.
\begin{prop}\label{tightness} Let $\eta>0$. Then the sequence $(\rho^{\eps})_{\eps>0}$ is tight in the spaces $C([0,T],H^{-\eta}(\T^N))$ and $L^2(0,T;L^2(\T^N))$.
\end{prop}
\begin{proof}

\noindent
\textit{Step 1: control of the modulus of continuity of $\re$ in $H^{-\eta}(\T^N)$.} Let $\eta>0$ be fixed. For any $\delta>0$, we define $$w(\rho,\delta):=\sup\limits_{|t-s|<\delta}\|\rho(t)-\rho(s)\|_{H^{-\eta}(\T^N)}$$
the modulus of continuity of a function $\rho\in C([0,T],H^{-\eta}(\T^N))$. In this first step of the proof, we want to obtain the following bound 
\begin{equation}\label{controlw}
\E w(\re,\delta)\lesssim \eps + \delta^{\tau},
\end{equation}
for some positive $\tau$. To do so, we use the perturbed test-functions method. Let $(p_j)_{j\in\N^N}$ the Fourier orthonormal basis of $L^2(\T^N)$ and $J$ the operator
$$J:=(\textrm{I}-\Delta_x)^{-\frac{1}{2}}.$$ Let $j\in\N^N$. We set $$\varphi_j(f):=(f,p_jF), \quad f\in\LL,$$ and we define the first order corrections by, see Section \ref{subsec:formaltest},
$$\varphi_{1,j}(f,n):=-(f\MI ,p_jF), \quad (f,n)\in\LL\times E.$$
We finally define $\varphi^{\eps}_j:=\varphi_j+\eps\varphi_{1,j}$, which is easily seen to be a good test-function, so that, thanks to Proposition \ref{gene}, we consider the continuous martingales
$$M^{\eps}_j(t):=\varphi^{\eps}_j(f^{\eps}_t,m^{\eps}_t)-\varphi^{\eps}_j(f^{\eps}_0,m^{\eps}_0)-\int_0^t\mathscr{L}^{\eps}\varphi^{\eps}_j(f^{\eps}_s,m^{\eps}_s)\,\dd s.$$
We also define,
$$\theta_j^{\eps}(t):=\varphi_j(f^{\eps}_0)+\int_0^t\mathscr{L}^{\eps}\varphi^{\eps}_j(f^{\eps}_s,m^{\eps}_s)\,\dd s+M_{j}^{\eps}(t).$$
Note that 
\begin{equation}\label{thetaj}
\theta_j^{\eps}(t)=\varphi_j(f^{\eps}_t)+\eps(\varphi_{1,j}(f^{\eps}_t,m^{\eps}_t)-\varphi_{1,j}(f^{\eps}_0,m^{\eps}_0)),
\end{equation}
so that, with the definitions of $\varphi_j$ and $\varphi_{1,j}$, Cauchy-Schwarz inequality, we easily get
$$|\theta_j^{\eps}(t)|\lesssim\sup\limits_{t\in[0,T]}\|f^{\eps}(t)\|\|p_j\|_{L^2_x}=\sup\limits_{t\in[0,T]}\|f^{\eps}(t)\|.$$
Hence, by the uniform $\LL$ bound $(\ref{L2bound})$,
\begin{equation}\label{bornethetajeps}\E\sup\limits_{t\in [0,T]}\left|\theta_j^{\eps}(t)\right|\lesssim 1.\end{equation}
With $(\ref{thetaj})$ and the uniform $\LL$ bound $(\ref{L2bound})$, we also deduce
\begin{equation}\label{bornethetajmoins}\E\sup\limits_{t\in [0,T]}\left|\varphi_j(\re_t)-\theta_j^{\eps}(t)\right|\lesssim \eps.\end{equation}
From now on, we fix $\gamma > N/2 + 2$ and we remark that, by $(\ref{bornethetajeps})$, a.s. and for all $t\in[0,T]$, the series defined by $u^{\eps}_t:=\sum_{j\in\N^N}\theta^{\eps}_j(t)J^{\gamma}p_j$ converges in $L^2(\T^N)$. We then set $$\theta^{\eps}(t):=J^{-\gamma}\sum\limits_{j\in\N^N}\theta^{\eps}_j(t)J^{\gamma}p_j,$$
which exists a.s. and for all $t\in[0,T]$ in $H^{-\gamma}(\T^N)$. And with $(\ref{bornethetajmoins})$, we obtain
\begin{equation}\label{bornethetamoins}\E\sup\limits_{t\in [0,T]}\left\|\re(t)-\theta^{\eps}(t)\right\|_{H^{-\gamma}(\T^N)}\lesssim \eps.\end{equation}
\noindent Actually, by interpolation, the continuous embedding $L^2(\T^N)\subset H^{-\eta}(\T^N)$ and the uniform $\LL$ bound $(\ref{L2bound})$, we have 
$$\E\sup\limits_{|t-s|<\delta}\|\rho(t)-\rho(s)\|_{H^{-\eta^{\flat}}}\leq\E\sup\limits_{|t-s|<\delta}\|\rho(t)-\rho(s)\|^{\upsilon}_{H^{-\eta^{\sharp}}}$$
for a certain $\upsilon>0$ if $\eta^{\sharp}>\eta^{\flat}>0$. As a result, it is indeed sufficient to work with $\eta=\gamma$. In view of $(\ref{bornethetamoins})$, we first want to obtain an estimate of the increments of $\theta^{\eps}$. We have, for $j\in\N^N$ and $0\leq s\leq t\leq T$,
\begin{equation}\label{accroisstheta}
\theta^{\eps}_j(t)-\theta^{\eps}_j(s)=\int_s^t\mathscr{L}^{\eps}\varphi_j^{\eps}(f^{\eps}_{\sigma},m^{\eps}_{\sigma})\,\dd \sigma+M^{\eps}_j(t)-M^{\eps}_j(s).
\end{equation}
We then control the two terms on the right-hand side of $(\ref{accroisstheta})$. Let us begin with the first one. Note that, since $D\varphi_j(f)\equiv p_j F$ and $D\varphi_{1,j}(f)\equiv -\MI p_j F$, we obtain thanks to \eqref{afterone} with $\varphi_2\equiv 0$,
$$\mathscr{L}^{\eps}\varphi^{\eps}_j(\fe_{\sigma},m^{\eps}_{\sigma})=-\frac{1}{\eps}(A\fe_{\sigma},p_jF) + (A\fe_{\sigma},M^{\!-\! 1}\!I(m^{\eps}_{\sigma}) p_jF) - (\fe_{\sigma}m^{\eps}_{\sigma},M^{\!-\! 1}\!I(m^{\eps}_{\sigma}) p_jF).$$
Since, with $(\ref{nullflux})$, we have $\li{a(v)\fe_{\sigma}}=\li{a(v)g^{\eps}_{\sigma}}$ where $g^{\eps}$ has been defined previously as $g^{\eps}:=\fe-\re F$, we can write
\begin{align*}
(A\fe_{\sigma},p_jF)&=\int_{\T^N}\mathrm{div}_x(\li{a(v)\fe_{\sigma}})p_j\,\dd x = \int_{\T^N}\mathrm{div}_x(\li{a(v)\gep_{\sigma}})p_j\,\dd x = (A\gep_{\sigma},p_jF)
\end{align*}
and, as a consequence, since $a$ is bounded, we are led to
$$\frac{1}{\eps}(A\fe_{\sigma},p_jF)\ \lesssim\ \|\eps^{-1}g^{\eps}_{\sigma}\|\|\nabla_xp_j\|_{L^2}. $$
Similarly, we can show that
$$(A\fe_{\sigma},M^{\!-\! 1}\!I(m^{\eps}_{\sigma}) p_jF)\ \lesssim\ \|g^{\eps}_{\sigma}\|(1+\|\nabla_xp_j\|_{L^2}).$$
Since we have obviously $(\fe_{\sigma}m^{\eps}_{\sigma},M^{\!-\! 1}\!I(m^{\eps}_{\sigma}) p_jF)\lesssim \|\fe_{\sigma}\|$, we can conclude that
\begin{equation}\label{controlL}
|\mathscr{L}^{\eps}\varphi^{\eps}_j(\fe_{\sigma},m^{\eps}_{\sigma})|\lesssim C_j \left[\|\eps^{-1}g^{\eps}_{\sigma}\|+\|g^{\eps}_{\sigma}\|+\|\fe_{\sigma}\|\right],
\end{equation}
where $C_j:=1+ \|\nabla_xp_j\|_{L^2} \leq 1+|j|$.
Thanks to $(\ref{L2bound})$ and $(\ref{gepsbound})$ with $p=4$, we have that $(\eps^{-1}g^{\eps})_{\eps>0}$, $(g^{\eps})_{\eps>0}$ and $(\fe)_{\eps>0}$ are bounded in $L^4(\Omega;L^2(0,T;\LL))$. As a consequence, $(\ref{controlL})$ and an application of H\"{o}lder's inequality gives
$$\E\left|\int_s^t\mathscr{L}^{\eps}\varphi_{j}^{\eps}(f^{\eps}_{\sigma},m^{\eps}_{\sigma})\,d\sigma\right|^4\lesssim C_j^4|t-s|^2.$$
Furthermore, using Burkholder-Davis-Gundy inequality, we can control the second term of the right-hand side of $(\ref{accroisstheta})$ as
$$\E|M^{\eps}_j(t)-M^{\eps}_j(s)|^4\lesssim \E|\langle M^{\eps}_j\rangle_t-\langle M^{\eps}_j\rangle_s|^2,$$
where the quadratic variation $\langle M^{\eps}_j\rangle$ is given by
$$\langle M^{\eps}_j\rangle_t=\int_0^t(M|\varphi_{1,j}|^2-2\varphi_{1,j}M\varphi_{1,j})(\fe_s,m^{\eps}_s)\,\dd s.$$
With the definition of $\varphi_{1,j}$, \eqref{mbound}, \eqref{MIcarre} and the uniform $\LL$ bound $(\ref{L2bound})$, it is now easy to get
$$\E|M^{\eps}_j(t)-M^{\eps}_j(s)|^4\lesssim |t-s|^2.$$
Finally we have $\E|\theta^{\eps}_j(t)-\theta^{\eps}_j(s)|^4\lesssim (1+|j|^4)|t-s|^2$. Since we took $\gamma>N/2+2$, we can conclude that 
$$\E\|\theta^{\eps}(t)-\theta^{\eps}(s)\|_{H^{-\gamma}(\T^N)}^4\lesssim |t-s|^2.$$
It easily follows that, for $\upsilon<1/2$,
$$\E\|\theta^{\eps}\|_{W^{\upsilon,4}(0,T,H^{-\gamma}(\T^N))}^4\lesssim 1$$
and by the embedding 
$$W^{\upsilon,4}(0,T,H^{-\gamma}(\T^N))\subset \mathcal{C}^{\tau}(0,T,H^{-\gamma}(\T^N)),\quad \tau<\upsilon-\frac{1}{4},$$
we obtain that $\E w(\theta^{\eps},\delta)\lesssim \delta^{\tau}$ for a certain positive $\tau$. Finally, with $(\ref{bornethetamoins})$, we can now conclude the first step of the proof since
\begin{equation}
\E w(\rho^{\eps},\delta)\leq 2\E\sup\limits_{t\in [0,T]}\left\|\rho^{\eps}_t-\theta^{\eps}_t\right\|_{H^{-\gamma}(\T^N)}+\E w(\theta^{\eps},\delta)\lesssim \eps + \delta^{\tau}.
\end{equation}

\textit{Step 2: tightness in $C([0,T];H^{-\eta}(\T^N))$.} Since the embedding $L^2(\T^N)\subset H^{-\eta}(\T^N)$ is compact, and by Ascoli's Theorem, the set 
$$K_R:=\left\lbrace\rho\in C([0,T],H^{-\eta}(\T^N)),\; \sup\limits_{t\in[0,T]}\|\rho\|_{L^2(\T^N)}\leq R,\; w(\rho,\delta)<\eps(\delta)\right\rbrace,$$
where $R>0$ and $\eps(\delta)\to 0$ when $\delta\to 0$, is compact in $C([0,T],H^{-\eta}(\T^N))$. By Prokohrov's Theorem, the tightness of $(\re)_{\eps>0}$ in $C([0,T],H^{-\eta}(\T^N))$ will follow if we prove that for all $\sigma>0$, there exists $R>0$ such that \begin{equation}\label{tight1}\PP(\sup\limits_{t\in[0,T]}\|\rho^{\eps}\|_{L^2(\T^N)}>R)<\sigma,
\end{equation}
and
\begin{equation}\label{tight2}\lim\limits_{\delta\to 0}\limsup\limits_{\eps\to 0}\PP(w(\rho^{\eps},\delta)>\sigma)=0.
\end{equation}
With Markov's inequality and the uniform $\LL$ bound $(\ref{L2bound})$, we have
$$\PP(\sup\limits_{t\in[0,T]}\|\rho^{\eps}\|_{L^2(\T^N)}>R)\ \leq\ \PP(\sup\limits_{t\in[0,T]}\|f^{\eps}\|>R)\ \lesssim\ R^{-1},$$
which gives $(\ref{tight1})$. And we deduce $(\ref{tight2})$ by Markov's inequality and the bound $(\ref{controlw})$ since 
\begin{align*}
\lim\limits_{\delta\to 0}\limsup\limits_{\eps\to 0}\PP(w(\rho^{\eps},\delta)>\sigma)&\ \leq\ \lim\limits_{\delta\to 0}\limsup\limits_{\eps\to 0}\sigma^{-1}\E w(\rho^{\eps},\delta)\\
&\ \lesssim\ \lim\limits_{\delta\to 0}\limsup\limits_{\eps\to 0}\sigma^{-1}(\eps + \delta^{\tau})\ =\ 0.
\end{align*}

\textit{Step 3: tightness in $L^2(0,T;L^2(\T^N))$.} Similarly, due to \cite[Theorem 5]{simon}, the set 
$$K_R:=\left\lbrace\rho\in L^2(0,T;L^2(\T^N)),\; \int_0^T\!\!\|\rho_t\|^2_{H^s(\T^N)}\dd t\leq R,\; w(\rho,\delta)<\eps(\delta)\right\rbrace,$$
where $R>0$, $s>0$ and $\eps(\delta)\to 0$ when $\delta\to 0$, is compact in $L^2(0,T;L^2(\T^N))$. By Prokhorov's Theorem, the  tightness of $(\re)_{\eps>0}$ in $L^2(0,T;L^2(\T^N))$ will follow if we prove that for all $\sigma>0$, there exists $R>0$ such that \begin{equation}\label{tight11}\PP(\int_0^T\!\!\|\rho_t\|^2_{H^s(\T^N)}\dd t>R)<\sigma,
\end{equation}
and
\begin{equation}\label{tight22}\lim\limits_{\delta\to 0}\limsup\limits_{\eps\to 0}\PP(w(\rho^{\eps},\delta)>\sigma)=0.
\end{equation}
But $(\ref{tight11})$ and $(\ref{tight22})$ are consequences of Markov's inequality and the bounds $(\ref{Hsbound})$ with $p=1$  and $(\ref{controlw})$ so that the proof is complete.
\end{proof}

\section{Convergence}

We conclude here the proof of Theorem \ref{mainresult}. The idea is now, by the tightness result and Prokhorov Theorem, to take a subsequence of $(\rho^{\eps})_{\eps>0}$ that converges in law to some probability measure. Then we show that this limiting probability is actually uniquely determined by the limit generator $\mathscr{L}$ defined above.\medskip

\noindent We fix $\eta>0$. By Proposition $\ref{tightness}$ and Prokhorov's Theorem, there is a subsequence of $(\rho^{\eps})_{\eps >0}$, still denoted $(\rho^{\eps})_{\eps >0}$, and a probability measure $P$ on the spaces $C([0,T],H^{-\eta})$ and $L^2(0,T;L^2)$ such that
$$P^{\eps}\to P\text{ weakly in } C([0,T],H^{-\eta}) \text{ and }L^2(0,T;L^2),$$
where $P^{\eps}$ stands for the law of $\rho^{\eps}$. We now identify the probability measure $P$.\medskip

\noindent Since the spaces $C([0,T],H^{-\eta})$ and $L^2(0,T;L^2)$ are separable, we can apply Skohorod representation Theorem \cite{billingsley}, so that there exists a new probability space $(\widetilde{\Omega},\widetilde{\mathcal{F}},\widetilde{\PP})$ and random variables $$\widetilde{\re},\widetilde{\rho}:\widetilde{\Omega}\to C([0,T],H^{-\eta})\cap L^2(0,T;L^2),$$ with respective law $P^{\eps}$ and $P$ such that $\widetilde{\rho^{\eps}}\to \widetilde{\rho}$ in $C([0,T],H^{-\eta})$ and $L^2(0,T;L^2)$ $\widetilde{\PP}-$a.s. In the sequel, for the sake of clarity, we do not write any more the tildes.\medskip

\noindent Note that, with $(\ref{gepsbound})$, we can also suppose that $\eps^{-1} g^{\eps}$ converges to some $g$ weakly in  the space $L^2(\Omega;L^2(0,T;\LL))$. Similarly, with \eqref{mbound}, we assume that $m^{\eps}$ converges to m weakly in $L^2(\Omega;L^2(0,T;\LL))$. Before going on the proof, we want to identify the weak limit $g$ of $\eps^{-1} g^{\eps}$.

\begin{lemma}\label{lemmag}In $L^2(\Omega;L^2(0,T;L^2))$, we have the relation
$$\li{a(v)g}=-\sigma(\rho)^{-1}K\nabla_x\rho.$$
\end{lemma}

\begin{proof} We define $D_T := (0,T)\times \T^N$. Since $\fe$ satisfies equation $(\ref{rt})$, we can write, for any $\psi\in C^{\infty}_c(D_T)$ and $\theta\in L^{\infty}(V\times\Omega;\R^N)$,
\begin{align*}
\E\int_{D_T\times V}\fe F^{-1}\left(-\eps\partial_t\psi-a\cdot\nabla_x\psi\right)\theta &\ =\ \E\int_{D_T\times V}\frac{1}{\eps}\sigma(\li{\fe})L\fe F^{-1}\psi\theta  \\
&+\ \E\int_{D_T\times V}m^{\eps}\fe F^{-1} \psi\theta.
\end{align*}
We recall that we set $\gep:=\fe-\re F$ and that $L\fe = L\gep$ so that we have
\begin{align*}
&\E\int_{D_T\times V}-\eps\fe F^{-1}\partial_t\psi\,\theta-\re a\cdot\nabla_x\psi\,\theta - \gep F^{-1}a\cdot\nabla_x\psi\, \theta \\
& \qquad =\  \E\int_{D_T\times V}\sigma(\re)L(\eps^{-1}\gep) F^{-1} \psi\theta\ +\ \E\int_{D_T\times V}m^{\eps}\fe F^{-1} \psi\theta.
\end{align*}
Since $(\fe)_{\eps>0}$ and $(\eps^{-1}\gep)_{\eps>0}$ are bounded in $L^2(\Omega;L^2(0,T;\LL))$ by \eqref{L2bound} and \eqref{gepsbound}, and with the $\PP-$a.s. convergence $\re\to\rho$ in $L^2(0,T;\LL)$ coupled with the uniform integrability of the family $(\re)_{\eps>0}$ obtained with \eqref{L2bound}, we have that the left-hand side of the previous equality actually converges as $\eps\to 0$ to
$$\E\int_{D_T\times V}-\rho a\cdot\nabla_x\psi\,\theta.$$
Note that, $\PP-$a.s., we have the following convergences in $L^2(0,T;\LL)$
$$\sigma(\re)\to \sigma(\rho),\quad L(\eps^{-1}\gep)\rightharpoonup Lg, \quad \fe\to \rho F, \quad m^{\eps} \rightharpoonup \textrm{m},$$ where the first convergence is justified by the Lipschitz continuity of $\sigma$. As a result, since all the quantities above are uniformly integrable with respect to $\eps$ thanks to \eqref{L2bound}, \eqref{gepsbound} and \eqref{mbound}, the right-hand side of the previous equality converges as $\eps\to 0$ to
$$\E\int_{D_T\times V}\sigma(\rho)L(g) F^{-1} \psi\theta\ +\ \E\int_{D_T\times V}\textrm{m}\rho \psi\theta.$$
Thus, we have
$$\E\int_{D_T\times V}-\rho a\cdot\nabla_x\psi\,\theta\ =\ \E\int_{D_T\times V}\sigma(\rho)L(g) F^{-1} \psi\theta\ +\ \E\int_{D_T\times V}\textrm{m}\rho \psi\theta.$$
Let $\xi$ be an arbitrary bounded measurable function on $\Omega$. We now set $\theta(v,\omega)=a(v)F(v)\xi(\omega)$; note that we do have $\theta\in L^{\infty}(V\times\Omega,\R^N)$. With $(\ref{nullflux})$ and the relation $Lg=\li{g}F-g$, we obtain
$$-\E\int_{D_T\times V}\rho a\cdot\nabla_x\psi\,aF\ =\ -\E\int_{D_T\times V}\sigma(\rho)g a(v)\psi.$$
Since this relation holds for any $\xi \in L^{\infty}(\Omega)$ and $\psi\in C^{\infty}_c(D_T)$, we deduce that $\nabla_x \rho \in L^2(\Omega,L^2(D_T))$ and that
$$\li{a(v)g}=-\sigma(\rho)^{-1}K\nabla_x\rho,$$
and this concludes the proof. 
\end{proof}

\noindent Let $\varphi\in C^3(\LL)$ a good test-function satisfying \eqref{smoothtest}. We define $\varphi^{\eps}$ as in Section \ref{subsec:formaltest}. Since $\varphi^{\eps}$ is a good test-function, we have that
$$\varphi^{\eps}(\fe_t,m^{\eps}_t)-\varphi^{\eps}(\fe_0,m^{\eps}_0)-\int_0^t\mathscr{L}^{\eps}\varphi^{\eps}(\fe_s,m^{\eps}_s)\,\dd s,\quad t\in [0,T],$$
is a continuous martingale for the filtration generated by $(\fe_s)_{s\in[0,T]}$. As a result, if $\Psi$ denotes a continuous and bounded function from $L^2(\T^N)^n$ to $\R$, we have
\begin{equation}\label{avantlimite}
\E\left[\left(\varphi^{\eps}(\fe_t,m^{\eps}_t)-\varphi^{\eps}(\fe_s,m^{\eps}_s)-\int_s^t\mathscr{L}^{\eps}\varphi^{\eps}(\fe_u,m^{\eps}_u)\,\dd u\right)\Psi(\re_{s_1},...,\re_{s_n})\right] = 0,
\end{equation}
for any $0\leq s_1\leq ...\leq s_n\leq s\leq t$. Our final purpose is to pass to the limit $\eps\to 0$ in \eqref{avantlimite}. In the sequel, we assume that the function $\varphi$ and $\Psi$ are also continuous on the space $H^{-\eta}$, which is always possible with an approximation argument: it suffices to consider $\varphi_r:=\varphi((\textrm{I}-r\Delta_x)^{-\frac{\eta}{2}}\cdot)$ and $\Psi_r:=\Psi((\textrm{I}-r\Delta_x)^{-\frac{\eta}{2}}\cdot,...,(\textrm{I}-r\Delta_x)^{-\frac{\eta}{2}}\cdot)$ as $r\to 0$. 
With \eqref{aftertwo}, we divide the left-hand side of \eqref{avantlimite} in four parts. Precisely, we define, for $i\in\{1,...,4\}$
\begin{align*}
&\tau^{\eps}_1:=\varphi^{\eps}(\fe_t,m^{\eps}_t)-\varphi^{\eps}(\fe_s,m^{\eps}_s),\\
& \tau^{\eps}_2:=\int_s^t\mathscr{L}\varphi(\re_u)\,\dd u,\\
& \tau^{\eps}_3:=\int_s^t-\frac{1}{\eps}(A\fe_u,D\varphi(\fe_u))-(\mathrm{div}_x(\sigma(\re_u)^{-1}K\nabla_x\re_u)F,D\varphi(\re_u F))\,\dd u,\\
& \tau^{\eps}_4:=\int_s^t-(A\fe_u,D\varphi_1(\fe_u))-\eps(A\fe_u,D\varphi_2(\fe_u))+\eps(\fe_um^{\eps}_u,D\varphi_2(\fe_u))\, \dd u.
\end{align*}

\textit{Study of $\tau^{\eps}_1$.} We recall that $\varphi^{\eps}(\fe_t,m^{\eps}_t)=\varphi(\re_tF)+\eps\varphi_1(\fe_t,m^{\eps}_t)+\eps^2\varphi_2(\fe_t,m^{\eps}_t)$ so that, with the $\PP-$a.s. convergence of $\re$ to $\rho$ in $C([0,T],H^{-\eta})$ and the bounds $(i)$ of \eqref{f1bounds} and \eqref{f2bounds}, we have that $\tau^{\eps}_1$ converges $\PP-$a.s. to $\varphi(\rho_t F)-\varphi(\rho_s F)$ as $\eps$ goes to $0$. Furthermore, with the continuity of $\Psi$ in $H^{-\eta}$, we also have that $\Psi(\re_{s_1},...,\re_{s_n})$ converges $\PP-$a.s. to $\Psi(\rho_{s_1},...,\rho_{s_n})$. Finally, since the family $\tau^{\eps}_1\Psi(\re_{s_1},...,\re_{s_n})$ is uniformly integrable with respect to $\eps$ thanks to \eqref{smoothtest}, the bounds $(i)$ of \eqref{f1bounds} and \eqref{f2bounds} and the uniform $\LL$ bound \eqref{L2bound}, we have that
$$\E[\tau^{\eps}_1\Psi(\re_{s_1},...,\re_{s_n})]\to \E\left[\left(\varphi(\rho_t F)-\varphi(\rho_s F)\right)\Psi(\rho_{s_1},...,\rho_{s_n})\right].$$

\textit{Study of $\tau^{\eps}_2$.} We recall, with \eqref{limitg}, that
\begin{align*}
\mathscr{L}\varphi(\re_u)=(\mathrm{div}_x(\sigma(\re_u)^{-1}K\nabla_x\re_u)F,D\varphi(\re_u F))-\int_E(\re_u F n\MI,D\varphi(\re_u F))\,d\nu(n)\\
-\int_ED^2\varphi(\re_u F)(\re_u F\MI,\re_u F n)\,d\nu(n).
\end{align*}
Thanks to the $\PP-$a.s. convergence of $\re$ to $\rho$ in $L^2(0,T;L^2)$ and with $\varphi\in C^3(\LL)$, we can pass to the limit $\eps\to 0$ in the term
$$\int_s^t\int_E-(\re_u F n\MI,D\varphi(\re_u F))-D^2\varphi(\re_u F)(\re_u F\MI,\re_u F n)\,d\nu(n)\,\dd u.$$
Regarding the first term of $\mathscr{L}\varphi(\re_u)$, we introduce
$$G(\rho):=\int_0^{\rho}\frac{\dd y}{\sigma(y)},$$
which is, thanks to the hypothesis (H1) made on $\sigma$, Lipschitz continuous on $L^2(\T^N)$. Now the first term of $\mathscr{L}\varphi(\re_u)$ writes
$$(\mathrm{div}_x(\sigma(\re_u)^{-1}K\nabla_x\re_u)F,D\varphi(\re_u F))=(\mathrm{div}_x\nabla_x G(\re_u)F,D\varphi(\re_u F)).$$
Furthermore, with \eqref{smoothtest}, the mapping $\rho\mapsto\partial^2_{x_i,x_j}D\varphi(\rho F)$ is continuous on $L^2(\T^N)$. As a result, we can now pass to the limit in the term
$$\int_s^t(\mathrm{div}_x(\sigma(\re_u)^{-1}K\nabla_x\re_u)F,D\varphi(\re_u F))\,\dd u.$$
To sum up, we obtain that $\tau^{\eps}_2$ converges $\PP-$a.s. to $\int_s^t\mathscr{L}\varphi(\rho_u)\,\dd u$ as $\eps$ goes to $0$. Finally, since the family $\tau^{\eps}_2\Psi(\re_{s_1},...,\re_{s_n})$ is uniformly integrable with respect to $\eps$ thanks to \eqref{smoothtest} and the uniform $\LL$ bound \eqref{L2bound}, we have that
$$\E[\tau^{\eps}_2\Psi(\re_{s_1},...,\re_{s_n})]\to \E\left[\left(\int_s^t\mathscr{L}\varphi(\rho_u)\,\dd u\right)\Psi(\rho_{s_1},...,\rho_{s_n})\right].$$

\textit{Study of $\tau^{\eps}_3$.} First of all, we observe that, with the decomposition $\fe=\re F+g^{\eps}$, \eqref{moyenformal} and \eqref{nullflux},
$$-\eps^{-1}(A\fe_u,D\varphi(\fe_u))=-\eps^{-1}(A\gep_u,D\varphi(\fe_u)),$$
so that, with the $\PP-$a.s. convergences in $L^2(0,T;L^2)$
$$\eps^{-1}g^{\eps}\rightharpoonup g,\quad \re\to\rho,$$
and the continuity of the mapping $\rho\mapsto AD\varphi(\rho F)$ thanks to \eqref{smoothtest}, we obtain that the first term of $\tau^{\eps}_3$ converges $\PP-$a.s. to 
$$-\int_s^t(\li{Ag_u}F,D\varphi(\rho_u F))\,\dd u.$$ 
And, with Lemma \ref{lemmag}, this term writes
\begin{equation}\label{fortau3}
\int_s^t(\mathrm{div}_x(\sigma(\rho_u)^{-1}K\nabla_x\rho_u)F,D\varphi(\rho_u F))\,\dd u.
\end{equation}
Furthermore, similarly as the case of $\tau_2^{\eps}$, we have that the second term of $\tau_{3}^{\eps}$ converges $\PP-$a.s. to the opposite of \eqref{fortau3}. As a result, $\tau_3^{\eps}$ converges $\PP-$a.s. to $0$. Finally, since the family $\tau^{\eps}_3\Psi(\re_{s_1},...,\re_{s_n})$ is uniformly integrable with respect to $\eps$ thanks to \eqref{smoothtest}, the uniform $\LL$ bound \eqref{L2bound} and the bound \eqref{gepsbound} on $(\eps^{-1}g^{\eps})_{\eps>0}$, we have that
$$\E[\tau^{\eps}_3\Psi(\re_{s_1},...,\re_{s_n})]\to 0.$$

\textit{Study of $\tau^{\eps}_4$.} If we transform the two first terms of $\tau_4^{\eps}$ exactly as we do for the first term of $\tau_3^{\eps}$, it is then easy, using the uniform bounds \eqref{L2bound} and \eqref{gepsbound} and the bounds $(ii)$ of \eqref{f1bounds} and \eqref{f2bounds}, to get
$$\E[\tau^{\eps}_4\Psi(\re_{s_1},...,\re_{s_n})]=O(\eps).$$
To sum up, we can pass to the limit $\eps\to 0$ in \eqref{avantlimite} to obtain
\begin{equation}
\E\left[\left(\varphi(\rho_t F)-\varphi(\rho_s F)-\int_s^t\mathscr{L}\varphi(\rho_u)\,\dd u\right)\Psi(\rho_{s_1},...,\rho_{s_n})\right] = 0.
\end{equation}
We recall that this is valid for all $n\in\N$, $0\leq s_1\leq...\leq s_n\leq s\leq t \in[0,T]$ and all $\Psi$ continuous and bounded function on $L^2(\T^N)^n$. Now, let $\xi$ be a smooth function on $L^2(\T^N)$. We choose $\varphi(f)=(f,\xi F)$. We can easily verify that $\varphi$ and $|\varphi|^2$ belong to $C^3(\LL)$ and that they are good test-function satisfying \eqref{smoothtest}. Thus, we obtain that
\begin{align*}
&N_t:=\varphi(\rho_t F)-\varphi(\rho_0 F)-\int_0^t\mathscr{L}\varphi(\rho_u)\,\dd u,\quad t\in[0,T],\\
&|\varphi|^2(\rho_t F)-|\varphi|^2(\rho_0 F)-\int_0^t\mathscr{L}|\varphi|^2(\rho_u)\,\dd u,\quad t\in[0,T],
\end{align*} 
are continuous martingales with respect to the filtration generated by $(\rho_s)_{s\in[0,T]}$. It implies (see appendix 6.9 in \cite{fgps}) that the quadratic variation of $N$ is given by
$$\langle N\rangle_t=\int_0^t\left[\mathscr{L}|\varphi|^2(\rho_u)-2\varphi(\rho_u)\mathscr{L}\varphi(\rho_u)\right]\,\dd u,\quad t\in[0,T].$$
Furthermore, we have
\begin{align*}
\mathscr{L}|\varphi|^2(\rho_u)-2\varphi(\rho_u)\mathscr{L}\varphi(\rho_u)&=-2\int_E(\rho_u F n,\xi F)(\rho_u F \MI,\xi F)\,\dd\nu(n)\\
&=2\E\int_0^{\infty}(\rho_u F m_0,\xi F)(\rho_u F m_t,\xi F)\,\dd t\\
&=\E\int_{\R}(\rho_u F m_0,\xi F)(\rho_u F m_t,\xi F)\,\dd t\\
&=\int_{\T^N}\int_{\T^N}\rho_u(x)\xi(x)\rho_u(y)\xi(y)k(x,y)\,\dd x \dd y \\
&=\|\rho_u Q^{\frac{1}{2}}\xi\|^2_{L^2}.
\end{align*}
This is valid for all smooth function $\xi$ of $L^2(\T^N)$ so we deduce that
$$M_t:=\rho_t-\rho_0-\int_0^t \mathrm{div}_x(\sigma(\rho_s)^{-1}K\nabla_x\rho_s)\,\dd s-\int_0^t \rho_s H\,\dd s,\quad t\in[0,T],$$
is a martingale with quadratic variation
$$\int_0^t \rho_s Q^{\frac{1}{2}}\left(\rho_sQ^{\frac{1}{2}}\right)^*\,\dd s.$$
Thanks to martingale representation Theorem, see \cite[Theorem 8.2]{daprato}, up to a change of probability space, there exists a cylindrical Wiener process $W$ such that 
$$\rho_t-\rho_0-\int_0^t\mathrm{div}_x(\sigma(\rho_s)^{-1}K\nabla_x\rho_s)\,\dd s-\int_0^t \rho_s H\,\dd s=\int_0^t\rho_s Q^{\frac{1}{2}}\,\dd W_{s}, \quad t\in[0,T].$$
This gives that $\rho$ has the law of a weak solution to the equation \eqref{stochasticeq} with paths in $C([0,T],H^{-\eta})\cap L^2(0,T;L^2)$. Since this equation has a unique solution with paths in the space $C([0,T],H^{-\eta})\cap L^2(0,T;L^2)$, and since pathwise uniqueness implies uniqueness in law, we deduce that $P$ is the law of this solution and is uniquely determined. Finally, by the uniqueness of the limit, the whole sequence $(P^{\eps})_{\eps>0}$ converges to $P$ weakly in the spaces of probability measures on $C([0,T],H^{-\eta})$ and $L^2(0,T;L^2)$. This concludes the proof of Theorem \ref{mainresult}.

\nocite{degond}
\nocite{debouard}

\bibliography{biblio}
\end{document}